\documentclass[10pt, a4paper]{article}

\usepackage{cite, amsmath, amssymb, booktabs, lastpage, graphicx, multirow}
\usepackage[hidelinks,hyperfootnotes=false]{hyperref}

\usepackage{geometry}
\usepackage{setspace}

\usepackage[mathscr]{euscript}
\usepackage{amsthm}
\usepackage{xcolor}
\theoremstyle{plain}
\newtheorem{lemma}{Lemma}
\newtheorem{theorem}{Theorem}

\newtheorem*{corollary}{Corollary}

\theoremstyle{remark}

\theoremstyle{definition}

\makeatletter
\renewcommand{\maketitle}{
\begin{center}

{\Large\bfseries \@title\par}
\vspace{6mm}

{\large\bfseries \@author\par}
\vspace{4mm}

{\itshape \@address\par}
\vspace{2mm}

{\small\ttfamily \@email\par}
\vspace{5mm}

\vspace{5mm}

\end{center}
}
\makeatother

\makeatletter
\newcommand{\address}[1]{\gdef\@address{#1}}
\newcommand{\email}[1]{\gdef\@email{#1}}
\makeatother

\address{}
\email{}

\usepackage{fancyhdr}

\usepackage[margin=1cm,%
			font=footnotesize,%
			format=hang,%
			labelsep=period,%
			labelfont=bf]{caption}
\allowdisplaybreaks

\title{On the Wiener and Harary Indices of Generalized Splitting and Shadow-Splitting Graphs}
\author{Purva J. Makadiya$^{a,}$\footnote{Corresponding author.}, Mahesh M. Jariya$^b$, Prashant J. Makadiya$^c$}
\address{$^{a,b}$Department of Mathematics, \\ Saurashtra University, Rajkot-360005, \\ Gujarat, India.
    }
\email{purvamakadiya2000@gmail.com, mahesh.jariya@gmail.com, prashantmakadiya1996@gmail.com}

\date{\today}

\begin{document}
\maketitle
\thispagestyle{empty}

\begin{abstract}
In this paper, we determine the Wiener index and the Harary index for the $(p,q)$-generalized splitting graph $S_{p,q}(G)$ and the $(c,k)$-shadow-splitting graph $H_{c,k}(G)$ for a connected graph $G$.
\end{abstract}

\vspace{2mm}

\vspace{1mm}

\onehalfspacing
\section{Introduction}
Within the discipline of chemical graph theory, molecular topologies are systematically characterized using numerical invariants termed topological indices. These mathematical descriptors are foundational for the development of QSPR and QSAR models, which facilitate the predictive estimation of physical, chemical, and biological properties of a compound \cite{topological_review}. 

One of the most important topological indices in this field is the Wiener index \cite{wiener1947}, introduced by Harry Wiener in 1947 to predict the boiling points of certain chemicals (alkanes). For a connected graph $G$, the Wiener index is defined as
\[
W(G)=\frac{1}{2}\sum_{u,v\in V(G)} d_G(u,v),
\]
where $d_G(u,v)$ denotes the length of a shortest path between the vertices $u$ and $v$. Equivalently, the Wiener index is the sum of the shortest path distances over all unordered pairs of vertices in $G$.

In 1993, Plav\v{s}i\'{c} et al. \cite{plavsic1993} and Ivanciuc et al. \cite{ivanciuc1993} independently introduced the Harary index, named after Frank Harary, which assigns a greater weight to closer vertices by summing reciprocal distances, defined as:
\[
H(G) = \frac{1}{2} \sum_{u \neq v} \frac{1}{d_G(u, v)}.
\]

Sampathkumar and Walikar \cite{sampathkumar1980} introduced the splitting graph $S(G)$, which is constructed by adding a new vertex $v'$ for every original vertex $v \in V(G)$ such that $v'$ is connected to the entire neighborhood of $v$. The Wiener and Harary indices for this particular class were recently evaluated by Campe\~{n}a et al. \cite{campena2022}.

More recently, M. E. Abdel-Aal \cite{abdel2013} generalized this concept by introducing the $q$-splitting graph, denoted as $S_q(G)$. This structure is constructed by taking $q$ disjoint copies of $G$ alongside a single set of independent splitting vertices. Each splitting vertex corresponds to a specific vertex $v$ in the original graph $G$. To form the edges, each splitting vertex is connected to all the neighbors of $v$ across every single one of the $q$ copies of $G$.

Extending this framework, Dudhat, Kaneria, and Popat introduced the $(p,q)$-generalized splitting graph, denoted as $S_{p,q}(G)$ \cite{dudhat2026}. For integers $p, q \ge 1$, this structure is constructed by taking $p$ disjoint copies of the graph $G$ alongside $q$ independent sets of splitting vertices. Each vertex within these splitting sets connects to its corresponding neighbors in all $p$ copies of the base graph, while the copies themselves remain entirely disjoint. 

In the same paper, Dudhat et al. \cite{dudhat2026} introduced another graph construction called the $(c,k)$-shadow-splitting graph, denoted by $H_{c,k}(G)$. Similarly to the generalized splitting graph, this structure is generated using $c$ disjoint copies of a base graph $G$ alongside $k$ independent sets of isolated vertices. However, the defining characteristic of this model is the incorporation of ``shadow edges.'' These supplementary edges connect different copies of the graph, ensuring that the distinct copies of $G$ are pairwise interconnected rather than remaining isolated components.

The main objective of this paper is to derive the Wiener and Harary indices for the $(p,q)$-generalized splitting graph $S_{p,q}(G)$ and the $(c,k)$-shadow-splitting graph $H_{c,k}(G)$ derived from any connected base graph $G$. Using a unified block distance matrix methodology, we derive exact analytical expressions for these indices. These formulations are expressed in terms of the original indices $W(G)$ and $H(G)$, the order of the graph $n$ and size $m$, and the number of edges $m^{\prime\prime}$ that do not participate in any triangular cycles.

The remainder of this paper is organized as follows: Section 2 outlines the fundamental concepts and notations utilized throughout the study. Section 3 investigates the topological indices for the generalized splitting graphs $S_{p,q}(G)$. Section 4 establishes the block distance matrix to derive the indices for the shadow-splitting graphs $H_{c,k}(G)$. Finally, Section 5 provides our concluding remarks.

\section{Preliminaries}
Throughout this paper, we consider only finite, simple, and connected graphs.
Consider a graph $G = (V, E)$ with vertex set $V$ of order $n$ and edge set $E$ of size $m$.
Let $m^{\prime}$ denote the number of edges lying in at least one triangle, and let $m^{\prime\prime}$ denote the number of edges that do not belong to any triangle. Two standard matrix representations are used in this work. First, the adjacency matrix $A(G) = [a_{ij}]$ \cite{Horn1991} is defined as a square matrix of dimension $n \times n$, where an entry $a_{ij}$ equals $1$ if an edge exists between vertices $v_i$ and $v_j$ (denoted $v_i \sim v_j$), and $0$ otherwise. Second, the shortest path distances are recorded in the distance matrix $D(G) = [d_{ij}]$, where $d_{ij} = d_G(v_i, v_j)$. Additionally, $I_n$ represents the identity matrix of order $n$.

For an arbitrary matrix $M$, let $\sum M$ denote the sum over all of its elements. Similarly, let $\sum\overline{M}$ denote the sum of the reciprocals of all non-zero elements within $M$.

Before proceeding to our main results, we recall the Wiener and Harary index formulas for some standard graph families \cite{campena2022} in Table \ref{tab:indices} below.

\begin{table}[htbp]
\centering
\renewcommand{\arraystretch}{1.5}
\begin{tabular}{lcc}
\hline
\textbf{Graph Family} & \textbf{Wiener Index $W(G)$} & \textbf{Harary Index $H(G)$} \\
\hline
\textbf{Complete graph} $K_n$ & $\frac{n(n-1)}{2}$ & $\frac{n(n-1)}{2}$ \\
\textbf{Star graph} $S_n$ & $(n-1)^2$ & $\frac{(n-1)(n+2)}{4}$ \\
\textbf{Complete bipartite graph} $K_{a,b}$ ($a+b=n$) & $a^2+ab+b^2-a-b$ & $\frac{a^2+b^2-a-b}{4}+ab$ \\
\textbf{Path graph} $P_n$ & $\frac{n(n^2-1)}{6}$ & $nH_{n-1}-n+1$ \\
\textbf{Cycle graph} $C_n$ & 
$\begin{cases} 
\frac{n^3}{8}, & \text{even } n \\ 
\frac{n(n^2-1)}{8}, & \text{odd } n 
\end{cases}$ & 
$\frac{1+(-1)^n}{2}+nH_{\lfloor(n-1)/2\rfloor}$ \\
\hline
\end{tabular}
\caption{Wiener and Harary indices for standard graph families.}
\label{tab:indices}
\end{table}

\vspace{0.5cm}
\noindent\textbf{Note on $P_n$:} A previous computation by Campe\~{n}a et al. \cite{campena2022} obtained the Harary index of the path graph as $nH_n - 1$. In this work, we utilize the corrected formulation $nH_{n-1} - n + 1$ originally established by Plav\v{s}i\'{c} et al. \cite{plavsic1993}, where $H_k = \sum_{i=1}^k \frac{1}{i}$ denotes the $k$-th harmonic number.

\section{Indices of the \texorpdfstring{$(p,q)$}{(p,q)}-Generalized Splitting Graphs \texorpdfstring{$S_{p,q}(G)$}{S\_\{p,q\}(G)}}
To evaluate the topological indices of the $(p,q)$-generalized splitting graph, we first analyze its vertex set, $V(S_{p,q}(G))$, which has order $(p+q)n$. We can logically decompose this set into $p$ distinct subsets $V^{(1)}, V^{(2)}, \dots, V^{(p)}$ (representing the disjoint copies of the base graph $G$) and $q$ distinct subsets $U^{(1)}, U^{(2)}, \dots, U^{(q)}$ (representing the isolated splitting vertices). 

Following the approach of Campe\~{n}a et al. \cite{campena2022}, we utilize an auxiliary distance matrix $W = [w_{ij}]_{n \times n}$ to map the shortest paths between the splitting vertices. The entries $w_{ij}$ are determined by the neighborhood intersections within the base graph $G$. Specifically, the diagonal entries are zero ($w_{ii} = 0$). For distinct vertices $v_i$ and $v_j$, if they share at least one common neighbor (i.e., $N_G(v_i) \cap N_G(v_j) \neq \emptyset$), the distance is $w_{ij} = 2$. If they are adjacent but possess no common neighbors, $w_{ij} = 3$. Under all other conditions, the distance resolves to the standard shortest path in the base graph, yielding $w_{ij} = d_G(v_i, v_j)$.

\begin{lemma}\label{lem:distance_blocks}
Let $D_{X,Y}$ represent the submatrix containing the distances between vertex sets $X$ and $Y$ in $S_{p,q}(G)$. By considering shortest paths, the distance matrix $D(S_{p,q}(G))$ is partitioned into the following block components:
\begin{itemize}
    \item \textbf{Intra-set distances:} $D_{V^{(l)}, V^{(l)}} = D(G)$ for all $1 \le l \le p$, and $D_{U^{(k)}, U^{(k)}} = W$ for all $1 \le k \le q$.
    \item \textbf{Inter-set distances:} $D_{V^{(l)}, V^{(h)}} = W + 2I_n$ for $l \neq h$, and $D_{U^{(k)}, U^{(h)}} = W + 2I_n$ for $k \neq h$.
    \item \textbf{Cross-set distances:} $D_{V^{(l)}, U^{(k)}} = D(G) + 2I_n$ for all $l$ and $k$.
\end{itemize}
\end{lemma}

\begin{theorem}\label{thm:wiener_spq}
For any connected graph $G$ of order $n$ and size $m$, and for any $p, q \geq 1$, the Wiener index of the generalized splitting graph $S_{p,q}(G)$ is given by:
$$W(S_{p,q}(G)) = (p+q)^2 W(G) + [p(p-1)+q^2](m+m^{\prime\prime}) + n(p+q)(p+q-1),$$
where $m^{\prime\prime}$ represents the number of edges in $G$ that do not belong to any triangle.
\end{theorem}

\begin{proof}
The Wiener index is half the sum of all entries in the distance matrix. Summing the entries of the block matrices established in Lemma \ref{lem:distance_blocks}, the total sum becomes:
\begin{align*}
\sum D(S_{p,q}(G)) &= p(2W(G)) + p(p-1)\left(\sum W + 2n\right) + q\sum W \\
&\quad + q(q-1)\left(\sum W + 2n\right) + 2pq(2W(G) + 2n) \\
&= 2p(1+2q)W(G) + [p(p-1)+q^2]\sum W + 2n(p+q)(p+q-1).
\end{align*}
Next, we relate the sum of the entries of $W$ to the parameters of $G$. The entries of $W$ are:
\[
w_{ij} = \begin{cases}
d_G(v_i, v_j) + 2 & \text{if } v_i \sim v_j \text{ and } N_G(v_i) \cap N_G(v_j) = \emptyset, \\
d_G(v_i, v_j) + 1 & \text{if } v_i \sim v_j \text{ and } N_G(v_i) \cap N_G(v_j) \neq \emptyset, \\
d_G(v_i, v_j) & \text{otherwise}.
\end{cases}
\]
Thus, the sum of all entries in $W$ is:
\[
\sum W = \sum D(G) + \sum_{v_i \sim v_j \text{ with common neighbor}} 1 + \sum_{v_i \sim v_j \text{ with no common neighbor}} 2.
\]
Let $m^{\prime}$ be the number of edges in $G$ whose endpoints share at least one common neighbor, and $m^{\prime\prime}$ be the number of edges whose endpoints share no common neighbor. Since each edge is counted twice in the sum, we have:
\[
\sum W = 2 W(G) + 2m^{\prime} + 4m^{\prime\prime} = 2 W(G) + 2(m^{\prime} + m^{\prime\prime}) + 2m^{\prime\prime} = 2 W(G) + 2m + 2m^{\prime\prime}.
\]
Substituting $\sum W = 2 W(G) + 2m + 2m^{\prime\prime}$ back into the sum gives:
\begin{align*}
\sum D(S_{p,q}(G)) &= 2p(1 + 2q) W(G) + [p(p-1) + q^2] (2 W(G) + 2m + 2m^{\prime\prime}) \\
&\quad + 2n(p+q)(p+q-1) \\
&= 2(p+q)^2 W(G) + 2[p(p-1) + q^2] (m + m^{\prime\prime}) \\
&\quad + 2n(p+q)(p+q-1).
\end{align*}
Dividing both sides by 2, we obtain the desired formula.
\end{proof}

\begin{theorem}\label{thm:harary_spq}
For any connected graph $G$ of order $n$ and size $m$, and for any $p, q \geq 1$, the Harary index of $S_{p,q}(G)$ is:
$${H}(S_{p,q}(G)) = (p+q)^2 H(G) - \frac{p(p-1)+q^2}{2}\left(m+\frac{1}{3}m^{\prime\prime}\right) + \frac{n}{4}(p+q)(p+q-1),$$ where $m^{\prime\prime}$ represents the number of edges in $G$ that do not belong to any triangle.
\end{theorem}

\begin{proof} The Harary index is half the sum of the reciprocals of all non-zero entries in the distance matrix. Summing the reciprocals of the non-zero entries in the block matrices described in Lemma \ref{lem:distance_blocks}, the total reciprocal sum becomes:
\begin{align*}
\sum \overline{D}(S_{p,q}(G)) &= p(2 H(G)) + p(p-1) \left(\sum \overline{W} + \frac{n}{2}\right) + q \sum \overline{W} \\
&\quad + q(q-1) \left(\sum \overline{W} + \frac{n}{2}\right) + 2pq \left(2 H(G) + \frac{n}{2}\right) \\
&= 2p(1 + 2q) H(G) + [p(p-1) + q^2] \sum \overline{W} \\
&\quad + \frac{n}{2}(p+q)(p+q-1).
\end{align*}
Now, we find $\sum \overline{W}$ by taking the reciprocals of the non-zero entries of $W$:
\begin{align*}
\sum \overline{W} &= \sum_{i \neq j} \frac{1}{w_{ij}} \\
&= \sum_{i \neq j} \frac{1}{d_G(v_i, v_j)} - \sum_{v_i \sim v_j} 1 + \sum_{v_i \sim v_j \text{ with common neighbor}} \frac{1}{2} \\
&\quad + \sum_{v_i \sim v_j \text{ with no common neighbor}} \frac{1}{3}.
\end{align*}
Expressing this in terms of $H(G)$, $m^{\prime}$, and $m^{\prime\prime}$, we get:
\begin{align*}
\sum \overline{W} &= 2 H(G) - 2m + m^{\prime} + \frac{2}{3} m^{\prime\prime} \\
&= 2 H(G) - m^{\prime} - \frac{4}{3} m^{\prime\prime}\\
&= 2 H(G) - m - \frac{1}{3} m^{\prime\prime}.
\end{align*}
Substituting this back into the total reciprocal sum:
\begin{align*}
\sum \overline{D}(S_{p,q}(G)) &= 2p(1 + 2q) H(G) + [p(p-1) + q^2] \left(2 H(G) - m - \frac{1}{3} m^{\prime\prime}\right) \\
&\quad + \frac{n}{2}(p+q)(p+q-1) \\
&= 2(p+q)^2 H(G) - [p(p-1) + q^2] \left(m + \frac{1}{3} m^{\prime\prime}\right) \\
&\quad + \frac{n}{2}(p+q)(p+q-1).
\end{align*}
Dividing both sides by 2, this completes the proof.
\end{proof}

\begin{corollary}
Let $G$ be a connected graph of order $n$ and size $m$. By setting $p = 1$ and $q = 1$, the $(p,q)$-generalized splitting graph $S_{p,q}(G)$ reduces to the standard splitting graph $S(G)$. Consequently, substituting $p = 1$ and $q = 1$ into Theorem \ref{thm:wiener_spq} and Theorem \ref{thm:harary_spq} directly yields the Wiener and Harary indices for the standard splitting graph.

For the Wiener index:
\begin{align*}
W(S(G)) &= (1+1)^2 W(G) + [1(1-1)+1^2](m+m^{\prime\prime}) + n(1+1)(1+1-1) \\
&= 4W(G) + m + m^{\prime\prime} + 2n.
\end{align*}

For the Harary index:
\begin{align*}
H(S(G)) &= (1+1)^2 H(G) - \frac{1(1-1)+1^2}{2}\left(m+\frac{1}{3}m^{\prime\prime}\right) + \frac{n}{4}(1+1)(1+1-1) \\
&= 4H(G) - \frac{1}{2}\left(m+\frac{1}{3}m^{\prime\prime}\right) + \frac{n}{2}.
\end{align*}
\end{corollary}

\begin{corollary}
Let $G$ be a connected graph of order $n$ and size $m$. By setting $p = 1$, the $(p, q)$-generalized splitting graph $S_{p,q}(G)$ reduces to the $q$-splitting graph $S_q(G)$. Substituting $p = 1$ into Theorem \ref{thm:wiener_spq} and Theorem \ref{thm:harary_spq} yields the exact Wiener and Harary indices for the $q$-splitting graph.

For the Wiener index:
\begin{align*}
W(S_{1,q}(G)) &= (1+q)^2 W(G) + [q(q-1)+q](m+m^{\prime\prime}) + n(1+q)(1+q-1) \\
&= (q+1)^2 W(G) + q^2(m+m^{\prime\prime}) + nq(q+1).
\end{align*}

For the Harary index:
\begin{align*}
H(S_{1,q}(G)) &= (1+q)^2 H(G) - \frac{q(q-1)+q}{2}\left(m+\frac{1}{3}m^{\prime\prime}\right) + \frac{n}{4}(1+q)(1+q-1) \\
&= (q+1)^2 H(G) - \frac{q^2}{2}\left(m+\frac{1}{3}m^{\prime\prime}\right) + \frac{nq(q+1)}{4}.
\end{align*}
\end{corollary}

\vspace{0.5cm}
Using the properties of triangle-free graphs (where $m^{\prime\prime} = m$) alongside the foundational indices detailed in Section 2, we can derive the specific indices for several standard graph families. These results are summarized in Table \ref{tab:generalized_splitting}.

\begin{table}[htbp]
\centering
\renewcommand{\arraystretch}{1.8}
\resizebox{\textwidth}{!}{
\begin{tabular}{ll}
\hline
\textbf{Base Graph} & \textbf{Topological Indices for $S_{p,q}(G)$} \\
\hline
\multirow{2}{*}{\textbf{Path} $P_n$ ($n \ge 3$)} & $W = (p+q)^2\left[\frac{n(n^2-1)}{6}\right] + 2[p(p-1)+q^2](n-1) + n(p+q)(p+q-1)$ \\
& ${H} = (p+q)^2(nH_{n-1}-n+1) - \frac{2}{3}[p(p-1)+q^2](n-1) + \frac{n}{4}(p+q)(p+q-1)$ \\
\hline
\multirow{2}{*}{\textbf{Cycle} $C_n$ ($n \ge 4$)} & $W = (p+q)^2W(C_n) + 2n[p(p-1)+q^2] + n(p+q)(p+q-1)$ \\
& ${H} = (p+q)^2H(C_n) - \frac{2n}{3}[p(p-1)+q^2] + \frac{n}{4}(p+q)(p+q-1)$ \\
\hline
\multirow{2}{*}{\textbf{Star} $S_n$ ($n \ge 3$)} & $W = (p+q)^2(n-1)^2 + 2[p(p-1)+q^2](n-1) + n(p+q)(p+q-1)$ \\
& ${H} = (p+q)^2\frac{(n-1)(n+2)}{4} - \frac{2}{3}[p(p-1)+q^2](n-1) + \frac{n}{4}(p+q)(p+q-1)$ \\
\hline
\multirow{2}{*}{\textbf{Bipartite} $K_{a,b}$} & $W = (p+q)^2(a^2+ab+b^2-a-b) + 2ab[p(p-1)+q^2] + n(p+q)(p+q-1)$ \\
& ${H} = (p+q)^2\left(\frac{a^2+b^2-a-b}{4}+ab\right) - \frac{2ab}{3}[p(p-1)+q^2] + \frac{n}{4}(p+q)(p+q-1)$ \\
\hline
\multirow{2}{*}{\textbf{Complete} $K_n$ ($n \ge 3$)} & $W = [(p+q)^2 + p(p-1) + q^2]\frac{n(n-1)}{2} + n(p+q)(p+q-1)$ \\
& ${H} = \left[(p+q)^2 - \frac{p(p-1)+q^2}{2}\right]\frac{n(n-1)}{2} + \frac{n}{4}(p+q)(p+q-1)$ \\
\hline
\end{tabular}
}
\caption{Wiener and Harary indices of generalized splitting graphs for specific families.}
\label{tab:generalized_splitting}
\end{table}

\section{Indices of the $(c,k)$-Shadow-Splitting Graphs \texorpdfstring{$H_{c,k}(G)$}{H\_\{c,k\}(G)}}
In this section, we turn our attention to the topological descriptors of the $(c,k)$-shadow-splitting graph, denoted by $H_{c,k}(G)$. The vertex set of this composite structure has order $(c+k)n$. To systematically analyze the distances, we partition the vertex set into two main groups: $c$ subsets $V^{(1)}, V^{(2)}, \dots, V^{(c)}$ corresponding to the copies of the base graph $G$, and $k$ subsets $U^{(1)}, U^{(2)}, \dots, U^{(k)}$ representing the isolated splitting vertices.
Utilizing the auxiliary matrix $W$ defined in Section 3, we can partition the complete distance matrix of $H_{c,k}(G)$ into distinct $n \times n$ submatrices.

\begin{lemma}\label{lem:shadow_blocks}
Let $D_{X,Y}$ denote the submatrix capturing the shortest path distances between an arbitrary vertex set $X$ and set $Y$ within $H_{c,k}(G)$. A case-by-case analysis of the shortest paths reveals the following block structure for $D(H_{c,k}(G))$:
\begin{itemize}
    \item \textbf{Base copy interactions:} Within the same copy, $D_{V^{(l)}, V^{(l)}} = D(G)$ for all $1 \le l \le c$. Across distinct copies, $D_{V^{(l)}, V^{(h)}} = D(G) + 2I_n$ for $l \neq h$.
    \item \textbf{Splitting set interactions:} Within the same splitting set, $D_{U^{(r)}, U^{(r)}} = W$ for all $1 \le r \le k$. Across distinct splitting sets, $D_{U^{(r)}, U^{(s)}} = W + 2I_n$ for $r \neq s$.
    \item \textbf{Cross-set interactions:} Between any base copy and splitting set, $D_{V^{(l)}, U^{(r)}} = D(G) + 2I_n$.
\end{itemize}
\end{lemma}

\begin{theorem}\label{thm:wiener_hck}
Let $G$ be a connected graph with order $n$ and size $m$. For any integers $c, k \ge 1$, the Wiener index of the $(c,k)$-shadow-splitting graph is evaluated as follows:
$$W(H_{c,k}(G)) = (c+k)^2 W(G) + k^2(m+m^{\prime\prime}) + n(c+k)(c+k-1),$$
where $m^{\prime\prime}$ denotes the quantity of edges in $G$ that are not part of any triangle.
\end{theorem}

\begin{proof} By definition, the Wiener index is half the sum of all entries in the distance matrix. Summing the entries of the distinct submatrices outlined in Lemma \ref{lem:shadow_blocks} yields:
\begin{align*}
\sum D(H_{c,k}(G)) &= c(2W(G)) + c(c-1)\left(2W(G) + 2n\right) + k\sum W \\
&\quad + k(k-1)\left(\sum W + 2n\right) + 2ck(2W(G) + 2n) \\
&= 2(c^2 + 2ck)W(G) + k^2\sum W + 2n[c(c-1) + k(k-1) + 2ck] \\
&= 2c(c+2k)W(G) + k^2\sum W + 2n(c+k)(c+k-1).
\end{align*}
Recall from Section 3 that the sum of all entries in $W$ is $\sum W = 2W(G) + 2m + 2m^{\prime\prime}$. Substituting this into the total sum and dividing by 2 yields the final expression.
\end{proof}

\begin{theorem}\label{thm:harary_hck}
Let $G$ be a connected graph with order $n$ and size $m$. For any integers $c, k \ge 1$, the Harary index of the $(c,k)$-shadow-splitting graph is:
$${H}(H_{c,k}(G)) = (c+k)^2 H(G) - \frac{k^2}{2}\left(m+\frac{1}{3}m^{\prime\prime}\right) + \frac{n}{4}(c+k)(c+k-1),$$
where $m^{\prime\prime}$ denotes the quantity of edges in $G$ that are not part of any triangle.
\end{theorem}

\begin{proof} The Harary index is half the sum of the reciprocals of the non-zero distance entries. Summing the reciprocals of the entries in the block structure from Lemma \ref{lem:shadow_blocks}, we obtain:
\begin{align*}
\sum \overline{D}(H_{c,k}(G)) &= c(2 H(G)) + c(c-1) \left(2 H(G) + \frac{n}{2}\right) + k \sum \overline{W} \\
&\quad + k(k-1) \left(\sum \overline{W} + \frac{n}{2}\right) + 2ck \left(2 H(G) + \frac{n}{2}\right) \\
&= 2(c^2 + 2ck) H(G) + k^2 \sum \overline{W} + \frac{n}{2} [c(c-1) + k(k-1) + 2ck] \\
&= 2c(c + 2k) H(G) + k^2 \sum \overline{W} + \frac{n}{2}(c+k)(c+k-1).
\end{align*}
Substituting the previously established identity $\sum \overline{W} = 2H(G) - m - \frac{1}{3}m^{\prime\prime}$ into this summation and dividing by 2 yields the desired theorem.
\end{proof}

\vspace{0.5cm}
Table \ref{tab:shadow_splitting} details the exact topological indices for several standard graph families. These formulas are derived by evaluating theorems \ref{thm:wiener_hck} and \ref{thm:harary_hck} under the triangle-free parameter ($m = m^{\prime\prime}$) using the initial data established in Section 2.

\begin{table}[htbp]
\centering
\renewcommand{\arraystretch}{1.8}
\resizebox{\textwidth}{!}{
\begin{tabular}{ll}
\hline
\textbf{Base Graph} & \textbf{Topological Indices for $H_{c,k}(G)$} \\
\hline
\multirow{2}{*}{\textbf{Path} $P_n$ ($n \ge 3$)} & $W = (c+k)^2\frac{n(n^2-1)}{6} + 2k^2(n-1) + n(c+k)(c+k-1)$ \\
& ${H} = (c+k)^2(nH_{n-1}-n+1) - \frac{2k^2}{3}(n-1) + \frac{n}{4}(c+k)(c+k-1)$ \\
\hline
\multirow{2}{*}{\textbf{Cycle} $C_n$ ($n \ge 4$)} & $W = (c+k)^2W(C_n) + 2k^2n + n(c+k)(c+k-1)$ \\
& ${H} = (c+k)^2H(C_n) - \frac{2k^2}{3}n + \frac{n}{4}(c+k)(c+k-1)$ \\
\hline
\multirow{2}{*}{\textbf{Star} $S_n$ ($n \ge 3$)} & $W = (c+k)^2(n-1)^2 + 2k^2(n-1) + n(c+k)(c+k-1)$ \\
& ${H} = (c+k)^2\frac{(n-1)(n+2)}{4} - \frac{2k^2}{3}(n-1) + \frac{n}{4}(c+k)(c+k-1)$ \\
\hline
\multirow{2}{*}{\textbf{Bipartite} $K_{a,b}$} & $W = (c+k)^2(a^2+ab+b^2-a-b) + 2k^2ab + n(c+k)(c+k-1)$ \\
& ${H} = (c+k)^2\left(\frac{a^2+b^2-a-b}{4}+ab\right) - \frac{2k^2}{3}ab + \frac{n}{4}(c+k)(c+k-1)$ \\
\hline
\multirow{2}{*}{\textbf{Complete} $K_n$ ($n \ge 3$)} & $W = [(c+k)^2 + k^2]\frac{n(n-1)}{2} + n(c+k)(c+k-1)$ \\
& ${H} = \left[(c+k)^2 - \frac{k^2}{2}\right]\frac{n(n-1)}{2} + \frac{n}{4}(c+k)(c+k-1)$ \\
\hline
\end{tabular}
}
\caption{Wiener and Harary indices of shadow-splitting graphs for specific families.}
\label{tab:shadow_splitting}
\end{table}




\section{Conclusion}
In this paper, we have studied the Wiener and Harary indices of $(p,q)$-generalized splitting graphs $S_{p,q}(G)$ and $(c,k)$-shadow-splitting graphs $H_{c,k}(G)$. By employing block decompositions of the distance matrices, we derived general closed-form expressions for the Wiener and Harary indices of these two constructions that hold for any connected graph $G$. We then applied these general theorems to obtain exact indices for paths, cycles, complete graphs, stars, and complete bipartite graphs. Future work may focus on investigating degree-based indices for these graph constructions.

\section*{Acknowledgements}
The first author acknowledges the financial support received from CSIR under the CSIR-Junior Research Fellowship (JRF) scheme.

\section*{Declarations}
\textbf{Conflict of interest:} The authors declare that they have no conflict of interest.

\end{document}